\documentclass[a4paper,reqno]{amsart}

\usepackage{amsthm}
\usepackage{amssymb}
\usepackage{amsmath}
\usepackage{mathtools}
\usepackage{pdfpages}
\usepackage{changes}
\usepackage{tikz-cd}
\usepackage{exscale,cite,color,amsopn}
\usepackage[colorlinks=true,pdftex,unicode=true,linktocpage,bookmarksopen,hypertexnames=false]{hyperref}

\renewcommand{\leq}{\leqslant}
\renewcommand{\geq}{\geqslant}

\DeclareMathOperator{\BrAut}{Aut_{Br}}
\DeclareMathOperator{\Cen}{Z}
\DeclareMathOperator{\Fix}{Fix}
\DeclareMathOperator{\id}{id}
\DeclareMathOperator{\Ker}{Ker}
\DeclareMathOperator{\Ret}{Ret}
\DeclareMathOperator{\Soc}{Soc}
\DeclareMathOperator{\Sym}{Sym}

\newcommand{\Aut}{\operatorname{Aut}}
\newcommand{\GAP}{\textsf{GAP}}

\makeatletter
\numberwithin{equation}{section}
\numberwithin{figure}{section}
\numberwithin{table}{section}
\newtheorem{thm}{Theorem}[section]
\newtheorem{lem}[thm]{Lemma}
\newtheorem{cor}[thm]{Corollary}
\newtheorem{pro}[thm]{Proposition}

\theoremstyle{definition} 
\newtheorem{defn}[thm]{Definition}

\newtheorem{rem}[thm]{Remark}
\newtheorem{exa}[thm]{Example}

\makeatother

\begin{document}

\title{Factorizations of skew braces}

\author{E. Jespers, {\L}. Kubat, A. Van Antwerpen, L. Vendramin}

\address[E. Jespers, {\L}. Kubat, A. Van Antwerpen]{Department of Mathematics, Vrije Universiteit Brussel, Pleinlaan 2, 1050 Brussel}
\email{eric.jespers@vub.be}
\email{lukasz.kubat@vub.be}
\email{arne.van.antwerpen@vub.be}

\address[L. Vendramin]{IMAS--CONICET and Departamento de Matem\'atica, FCEN, Universidad de Buenos Aires, Pabell\'on~1,
Ciudad Universitaria, C1428EGA, Buenos Aires, Argentina; and NYU-ECNU Institute of Mathematical Sciences at NYU Shanghai,
3663 Zhongshan Road North, Shanghai, 200062, China}
\email{lvendramin@dm.uba.ar}


\begin{abstract}
We introduce strong left ideals of skew braces and prove that they produce non-trivial decomposition of set-theoretic solutions
of the Yang--Baxter equation. We study factorization of skew left braces through strong left ideals and we prove analogs
of It\^{o}'s theorem in the context of skew left braces. As a corollary, we obtain applications to the retractability problem
of involutive non-degenerate solutions of the Yang--Baxter equation. Finally, we classify skew braces that contain
no non-trivial proper ideals. 
\end{abstract}

\maketitle


\section*{Introduction}

An important family of finite set-theoretic solutions of the Yang--Baxter equation is that of involutive non-degenerate multipermutation 
solutions. Such solutions first appeared in the work~\cite{MR1722951} of Etingof, Schedler, and Soloviev as generalizations of Lyubachenko's permutation
solutions. Now, these solutions appear in many different contexts: it is known, for example, that a finite involutive non-degenerate solution is 
multipermutation if and only if its structure group is left orderable~\cite{MR3815290,MR3572046,MR2189580}. The structure group of a finite solution
has a particularly important finite quotient known as the permutation group of the solution. It is natural to ask which group-theoretical properties of 
permutation groups detect multipermutation solutions. There are strong results in this
directions~\cite{MR2652212,MR3861714,MR2885602,MR3935814,MR2278047,MR3814340}. However, to understand multipermutation solutions it is not enough to 
know the group structure of their permutation groups. One really needs to understand deeply the permutation group of a solution, meaning that one needs 
to understand this as a brace. Braces (and more generally skew braces) are generalizations of radical rings that turn out to provide the right
algebraic framework to study set-theoretic solutions of the Yang--Baxter equation~\cite{MR3647970,MR2278047}. Skew braces are intensively studied, as 
they are known to have connections to several different topics~\cite{MR3835326,MR3177933,MR3834774,MR3291816,MR3763907}. The structure group and the 
permutation group of a solution are examples of skew braces. So these groups are not merely groups: there is some deep ring-theoretic information
hidden behind these groups associated with set-theoretic solutions. 

In this work, we study involutive non-degenerate (multipermutation) solutions by means of the skew brace structure of their permutation groups and
via factorizations of this skew brace. Our first main result is a sort of analog of It\^{o}'s celebrated theorem on metabelian groups, but now in the 
context of skew braces. As an application, we understand better involutive non-degenerate (multipermutation) solutions. We also show that in some sense 
our It\^{o}'s theorem cannot be improved naively, meaning that we cannot expect a result similar to that of Kegel--Wielandt on products of nilpotent 
groups being solvable. It is interesting to remark that recently Sysak asked about extending the results of factorization of groups to skew braces.
We know precisely which is the setting in which the results of factorization can be studied in the skew brace situation. This is why we introduce
strong left ideals. We show that these strong left ideals are related to non-trivial decomposability of solutions. Decomposable solutions also were 
introduced by Etingof, Schedler, and Soloviev in~\cite{MR1722951}, the same paper where multipermutation solutions were first considered. With concrete 
examples, we show that our method for studying multipermutation solutions gives deeper insights. Indeed, there are solutions where group theory is not 
enough to detect multipermutability and where this property can be easily recognized by means of the brace factorization. 

As another application of skew brace factorization we obtain several results on the structure of skew left braces. In the final main result
we consider characteristic ideals of skew left braces and we characterise skew left braces that are characteristically simple. They turn out
to be factorizable into copies of a simple skew left subbrace.

The paper is organized as follows. In Section~\ref{strong} we review some basics on skew left braces and we introduce strong left ideals. In 
Proposition~\ref{pro:decomposable} we show that strong left ideals yield non-trivial decomposable solutions of the Yang--Baxter equation.
In Section~\ref{factorizations} we introduce factorizations of skew left braces and prove our main results in Theorems~\ref{thm:Ito2}
and~\ref{thm:Itocor}. These results can be seen as analogs of well-known theorems of It\^{o}. In Corollary~\ref{cor:multiperm} we apply
these results to set-theoretic involutive non-degenerate solutions of the Yang--Baxter equation. In Section~\ref{characteristic_simple}
we introduce characteristic ideals of skew braces and we classify finite skew left braces that do not contain non-trivial characteristic ideals. 

\section{Skew braces and strong left ideals}
\label{strong}

A \emph{skew left brace} is a triple $(A,+,\circ)$, where $(A,+)$ and $(A,\circ)$ are groups and, for all $a,b,c\in A$, the following
compatibility condition holds \[a\circ(b+c)=a\circ b-a+a\circ c.\] The group $(A, +)$ is called the \emph{additive group} of $A$ and
$(A,\circ)$ is called the \emph{multiplicative group} of $A$. For $a\in A$ we denote by $a'$ the inverse of $a$ with respect to the circle
operation $\circ$. By convention, left braces will be those skew left braces with abelian additive group. A skew left brace $A$ is said
to be \emph{trivial} if both operations $+$ and $\circ$ coincide, i.e., if $a+b=a\circ b$ for all $a,b\in A$. If $A$ is a skew left brace,
then the map $\lambda\colon (A,\circ)\to\Aut(A,+)$, $a\mapsto\lambda_a$, where $\lambda_a(b)=-a+a\circ b$, is a group homomorphism. It follows that 
\begin{equation}
    \label{eq:formulas}
    a\circ b=a+\lambda_a(b),\quad
    a+b=a\circ\lambda^{-1}_a(b),\quad
    \lambda_a(a')=-a
\end{equation}
for all $a,b\in A$. 

For elements $a$ and $b$ of a skew left brace $A$ we put
\[
a * b = \lambda_a(b)-b=-a+a\circ b-b.
\]
The operation $*$ measures the ``difference'' between the additive and multiplicative operations. In particular, if for all $a,b\in A$ we have 
$a*b=0$, then both operations are the same, i.e., the skew left brace simply reduces to a group; and thus for braces this is an abelian group.

\begin{lem} 
\label{lem:calcbraces}
Let $A$ be a skew left brace. For any $x,y,z \in A$ the following statements hold:
\begin{enumerate}
    \item $x*(y+z) = x*y +y +x*z -y$,
    \item $(x\circ y)*z = x*(y*z) + y*z + x*z$.
\end{enumerate}
\end{lem}

\begin{proof}
As $\lambda_x\colon (A,+)\to(A,+)$ is a group homomorphism, it follows that
\begin{align*}
    x*(y+z) & =\lambda_x(y+z)-z-y\\
            & =\lambda_x(y) +\lambda_x(z)-z-y\\
            & =x*y + y + x*z -y.
\end{align*}
Since $\lambda\colon (A,\circ)\to\Aut(A,+)$ is a group homomorphism, 
\begin{align*}
    (x\circ y)*z & =\lambda_{x\circ y}(z)-z\\
                 & =\lambda_x(\lambda_y(z)) -z\\
                 & =\lambda_x(y*z) +\lambda_x(z)-z\\
                 & = x*(y*z) +y*z + x*z.\qedhere
\end{align*}
\end{proof}

Note that if we replace both $+$ and $\circ$ by a single operation, and $*$ by the ordinary commutator corresponding to this new operation, then in the 
previous lemma we obtain well-known commutator formulas from group theory. Thus, intuitively, one can understand this operation as an analog of the 
group-theoretical commutator. In this spirit, a trivial skew left brace corresponds with an ``abelian structure''. As we will show later, this inspires 
a counterpart to the It\^{o} theorem on metabelian groups in the context of skew left braces. However, this analogy fails for other properties defined 
by commutators (e.g., nilpotency, solvablility, \ldots). 

\begin{rem}
\label{rem:conjugation}
Using the formulas of~\eqref{eq:formulas} we obtain that 
\begin{align*}
    \label{eq:conjugation}
    a\circ b\circ a' & =a+\lambda_a(b\circ a')\\
                     & =a+\lambda_a(b+b*a'+a')\\
                     & =a+\lambda_a(b+b*a')-a
\end{align*}
for all $a,b\in A$. 
\end{rem}

A \emph{left ideal} of a skew left brace $A$ is a subgroup $I$ of $(A,+)$ such that
$\lambda_a(I)\subseteq I$ for all $a\in A$, which is equivalent to $a*b\in I$ for all $a\in A$ and $b\in I$. It follows that a left ideal is a
skew left subbrace and, in particular, $(I,\circ)$ is a subgroup of $(A,\circ)$. Moreover, if $(I,\circ)$ is normal in $(A,\circ)$ and $(I,+)$
is normal in $(A,+)$ then one says that $I$ is an \emph{ideal} of $A$. It is known that ideals of skew braces correspond bijectively to kernels
of skew brace homomorphisms. The socle $\Soc(A)$ of a skew left brace $A$ is defined as \[\Soc(A)=\Ker\lambda\cap\Cen(A,+)\] and it is an ideal of $A$. 

\begin{defn}
Let $(A,+,\circ)$ be a skew left brace. A left ideal $I$ is called a \emph{strong left ideal} if $(I,+)$ is a normal subgroup of $(A,+)$. 
\end{defn}

\begin{exa}
A characteristic subgroup of the additive group of a skew left brace is a strong left ideal. Furthermore, ideals are strong left ideals.
\end{exa}

Braces were introduced by Rump in~\cite{MR2278047} as an algebraic tool to study involutive non-degenerate set-theoretic solutions
of the Yang--Baxter equation. A \emph{set-theoretic solution} of the Yang--Baxter equation is a pair $(X,r)$, where $X$ is a set and
$r\colon X\times X\to X\times X$ is a bijective map such that
\[(r\times{\id})({\id}\times r)(r\times{\id})=({\id}\times r)(r\times{\id})({\id}\times r).\]
If we write $r(x,y)=(\sigma_x(y),\tau_y(x))$, then $(X,r)$ is said to be \emph{non-degenerate} 
if all maps $\sigma_x,\tau_x\colon X\to X$ are bijective. 

A set-theoretic solution $(X,r)$ is called \emph{decomposable} if there exists a non-trivial partition $X=Y\cup Z$ such that
$(Y,r|_{Y\times Y})$ and $(Z,r|_{Z\times Z})$ are set-theoretic solutions and $r(Y\times Z)=Z\times Y$ and $r(Z\times Y)=Y\times Z$. 

Skew left braces produce non-degenerate solutions of the Yang--Baxter equation: if $A$ is a skew left brace, then the pair $(A,r_A)$, where 
\[r_A\colon A\times A\to A\times A,\quad r_A(a,b)=(\lambda_a(b),\lambda_a(b)'\circ a\circ b),\] is a non-degenerate solution of the Yang--Baxter 
equation. Moreover, $(A,r_A)$ is involutive (i.e., $r_A^2=\id$) if and only if the group $(A,+)$ is abelian.  

The set-theoretic solution $(A,r_A)$ associated to a non-zero skew left brace $A$ is always decomposable as $A=\{0\}\cup A\setminus\{0\}$. 
We show that strong left ideals of $A$ provide more decompositions of $(A,r_A)$.

\begin{pro}
\label{pro:decomposable}
Let $A$ be a skew left brace. If there exists a proper strong left ideal $I$, then $(A,r_A)$ is decomposable as $A = I \cup A\setminus I$. 
\end{pro}

\begin{proof}
Let $I$ be a proper strong left ideal. As $I$ is a left ideal, it holds, for any $x \in A$, that $\lambda_x(I) = I$. Moreover,
$\lambda_x(A\setminus I) = A \setminus I$. Let $x \in A$ and $a \in I$. Let $\rho_x(a) = \lambda_a(x)'\circ a\circ x$. Then 
\[
\rho_x(a) = \lambda_a(x)'\circ a\circ x= (a'\circ\lambda^{-1}_{a'}(x))'\circ x=(a'+x)'\circ x.
\]
As $I$ is normal in $(A,+)$ this is equal to $(x + b)'\circ x$ for some $b \in I$. As $I$ is a left ideal, there exists $c\in I$ such
that the previous is equal to  \[( x \circ c)' \circ x = c' \circ x' \circ x = c' \in I.\]
Hence $\rho_x(I) = I$. By~\cite[Lemma 2.4]{MR3835326}, $\rho_x(A\setminus I) = A \setminus I$ and hence the claim follows. 
\end{proof}

\section{Skew left braces admitting a factorization}
\label{factorizations}

In this section we study skew left braces $A$ such that $A=B+C$ for left ideals $B$ and $C$. Observe that
in this case it follows from ~\eqref{eq:formulas} that 
\[
A=B+C=C+B=B\circ C=C\circ B.
\]

\begin{defn}
	Let $A$ be a skew left brace and let $B$ and $C$ be left ideals of $A$. We say
	that $A$ admits a \emph{factorization} through $B$ and $C$ if $A=B+C$.
\end{defn}

We will study the case where $B$ and $C$ are trivial skew left braces. For that purpose,  
we shall need the following lemma. 

\begin{lem}\label{lem:calculations}
    Let $A$ be a skew left brace such that $A = B+C$, where $B$ and $C$ are left ideals. If $B$ and $C$ are trivial as skew left braces then,
    for any $b,\beta \in B$ and $c,\gamma \in C$ and $a \in A$, the following statements hold:
	\begin{enumerate}
	    \item $\lambda_{\beta \circ \gamma} = \lambda_{\gamma \circ \beta}$,
	    \item $(c + b) \circ \beta - \beta = c + b + c* \beta $,
	    \item $b\circ c \circ b' \circ c' = b \circ c - c\circ b = b + \lambda_b(c) - \lambda_c(b) -c \in \Ker \lambda$.
	\end{enumerate}
\end{lem}
\begin{proof}
	To prove (1) put $c_1=\lambda_\beta(c)\in C$ and $b_1=\lambda_{\gamma}(b)\in B$. Then, as $B$ and $C$ are trivial skew left braces,
	\begin{align*}
		\lambda_\beta(b+c) & =\lambda_{\beta}(b)+\lambda_\beta(c)=b+c_1,\\
		\lambda_\gamma(b+c) & =\lambda_{\gamma}(b)+\lambda_{\gamma}(c)=b_1+c.
	\end{align*}
	Hence,
	\begin{align*}
		\lambda_{\beta\circ\gamma}(b+c)=b_1+c_1=\lambda_{\gamma\circ\beta}(b+c).
	\end{align*}
	Let us prove (2). As $B$ is a trivial skew left brace, it follows from~\eqref{eq:formulas} that 
	\begin{align*}
		(c+b)\circ\beta-\beta &= (c\circ\lambda_{c'}(b))\circ\beta-\beta\\
		&=c\circ(\lambda_{c'}(b)\circ\beta)-\beta\\
		&=c\circ(\lambda_{c'}(b)+\beta)-\beta\\
		&=c\circ\lambda_{c'}(b)-c+c\circ\beta-\beta\\
		&=c+b+c*\beta.
	\end{align*}
	Part (3) follows from the following computation
	\begin{align*}
	    b \circ c \circ b' \circ c' &= (b\circ c) +\lambda_{b\circ c} (b' + \lambda_{b'}(c'))\\
	    &=b + \lambda_b(c) + \lambda_{b\circ c}(b') + \lambda_{b \circ c \circ b'}(c')\\
	     &= b+\lambda_{b}(c) + \lambda_{c}(b') + \lambda_{b\circ b'}(c')\\ 
	                                & =b + \lambda_b(c) +\lambda_c(-b) -c\\
	                                & =b \circ c - c \circ b.
	\end{align*}
	Moreover, by (1) it follows that $b \circ c \circ b' \circ c' \in \Ker \lambda$.
\end{proof}


If $X$ and $Y$ are non-empty subsets of a skew left brace $A$, we define $X*Y$ as the additive subgroup of $A$ generated
by all elements of the form $x*y$, where $x\in X$ and $y\in Y$. One defines $A^{(1)} = A$ and $A^{(n)}=A^{(n-1)}*A$ for $n\geq 2$. Then 
\[
A^{(1)}\supseteq A^{(2)}\supseteq A^{(3)}\supseteq\dotsb
\]
is a chain of ideals of $A$ known as the \emph{right series} of $A$, see~\cite[Proposition 2.1]{CSV}. Following Rump~\cite{MR2278047}, 
the skew left brace $A$ is said to be \emph{right nilpotent} of class $m$ if $A^{(m)}=0$ and $A^{(m-1)}\ne 0$.

\begin{defn}
A skew left brace $A$ is said to be \emph{meta-trivial} if $A^{(2)}$ is a trivial skew left brace. Equivalently, there exists
an ideal $I$ of $A$ such that $I$ and $A/I$ are trivial as skew left braces. 
\end{defn}

A left ideal $I$ of a skew left brace $A$ is said to be  \emph{meta-trivial} if $I$ is meta-trivial as a skew left brace.

\begin{lem}\label{lem:hardworkfactoriz}
Let $A$ be a skew left brace such that $A=B+C$ is a factorization through left ideals  $B$ and $C$. If $B$ and $C$ are trivial skew left braces, then: 
\begin{enumerate}
    \item $B*C$ and $C*B$ are strong left ideals of $A$,
    \item $B*C$ and $C*B$ are trivial skew left braces, and 
    \item $A^{(2)} = C*B+B*C=B*C+C*B$.
\end{enumerate}
\end{lem}

\begin{proof}
Since $C$ is a left ideal, it follows that $B*C \subseteq C$. Let $b,\beta \in B$ and $c, \gamma \in C$. As $C$ is trivial, it follows that 
\begin{align*}
    \lambda_{b \circ c} (\beta * \gamma) & =\lambda_b(\beta * \gamma)\\
                                         & =\lambda_b \lambda_{\beta} (\gamma) - \lambda_b(\gamma)\\
                                         & =\lambda_{b \circ \beta \circ b'}\lambda_b(\gamma) -\lambda_b(\gamma)\\
                                         & =(b \circ \beta \circ b') * \lambda_b(\gamma) \in B*C. 
\end{align*}
Hence $B*C$ is a left ideal and trivial as a skew left brace. 

Let $a \in A$, $b \in B$ and $c \in C$. Write $a=b_1 + c_1$, with $b_1 \in B$ and $c_1 \in C$. Then, by Lemma \ref{lem:calcbraces},
\begin{equation}\label{eq:7}
    \begin{aligned}
        a+(b*c)-a & =a + \lambda_b(c) - c - a\\
                  & =-(b*a) + b*(a+c)\\
                  & =-(b*(b_1+c_1)) + b* (b_1 + c_1+c).
    \end{aligned}
\end{equation}
Now, as $B+C=C+B$, it follows that for any $\beta \in B$ and $\gamma \in C$, there exist $\beta_1\in B$ and $\gamma_1\in C$ such that
$\beta + \gamma = \gamma_1 + \beta_1$. Hence, for any $b \in B$ it holds, by Lemma \ref{lem:calcbraces}, that
\[b*(\beta + \gamma)=b*(\gamma_1 + \beta_1)= b*\gamma_1 + \gamma_1 + b*\beta_1 - \gamma_1=b*\gamma_1 ,\]
as $B$ is trivial. Applying this on \eqref{eq:7} it follows that $B*C$ is a normal subgroup of $(A,+)$.
This proves (1) and (2) for $B*C$. The proof for $C*B$ is similar.

Now we show that $A^{(2)} \subseteq C*B + B*C$. Let $b,b_1 \in B$ and $c,c_1 \in C$. Then, by Lemma \ref{lem:calcbraces},
\begin{align*}
    (b\circ c)*(b_1+c_1) & =(b\circ c)*b_1 + b_1 + (b \circ c)*c_1 -b_1\\
                         & =\lambda_{b\circ c}(b_1) - b_1 + b_1 + b*(c*c_1) + c*c_1 + b*c_1 -b_1\\
                         & =\lambda_c(b_1) - b_1 +b_1 +b*c_1-b_1\\
                         & =c*b_1 + b_1 + b*c_1 -b_1 \in C*B+B*C.
\end{align*}
Clearly $C*B + B*C \subseteq A^{(2)}$ and thus $A^{(2)} = C*B+B*C=B*C + C*B$. 
\end{proof}

The possible approach to skew left brace factorization is through strong left ideals. 
In this setting we prove an analog of It\^{o}'s theorem~\cite{MR0071426} for skew left braces. 

\begin{thm} 
\label{thm:Ito2}
Let $A$ be a skew left brace. If $A=B+C$ is a factorization through strong left ideals $B$ and $C$ that are trivial as skew left braces,
then $A$ is right nilpotent of class at most three. In particular, $A$ is meta-trivial.
\end{thm}

\begin{proof}
By Lemma \ref{lem:hardworkfactoriz}, $B*C$ and $C*B$ are strong left ideals of $A$, and both are trivial as skew left braces. Furthermore, 
\[
A^{(2)} = B*C + C*B = (B*C) \circ (C*B).
\]
It rests to show that $A^{(2)}$ acts trivially on $A$. We first show that $B*C$ acts trivially on $A$. For that purpose, 
let $b\in B$, $c\in C$ and $a\in A$. Write $a=\beta+\gamma$, where $\beta\in B$ and $\gamma\in C$. Then, again by Lemma~\ref{lem:calcbraces},
\[
(b*c)*(\beta+\gamma)=(b*c)*\beta+\beta+(b*c)*\gamma-\beta=(b*c)*\beta, 
\]
as $C$ is a trivial skew left brace. By Lemma \ref{lem:calculations}(3), 
\[
(b\circ c - c\circ b) + \beta = (b\circ c - c \circ b) \circ \beta = (b + \lambda_b(c) - \lambda_c(b) - c) \circ \beta.
\]
Since $(B,+)$ is a normal subgroup of $(A,+)$, 
\[
b\circ c-c\circ b=b+\lambda_b(c)-\lambda_c(b)-c=\lambda_b(c)-c+b_1
\]
for some $b_1\in B$. By Lemma~\ref{lem:calculations}(2), 
\begin{align*}
    (b\circ c - c\circ b) + \beta & =(\lambda_b(c) - c + b_1) \circ \beta\\
                                  & =\lambda_b(c) - c + b_1 + (b*c)*\beta + \beta
\end{align*}
and therefore $(b*c)*\beta = 0$. Thus $B*C$ acts trivially on $A$. As $(C,+)$ also is a normal subgroup of $(A,+)$,
it follows by symmetry that $C*B$ acts trivially on $A$. Hence $A^{(2)}$ acts trivially on $A$.
\end{proof}

\begin{cor}
    Let $A$ be a skew left brace. Assume that $A= B + C$, where $B$ and $C$ are (not necessarily strong) left ideals,
    which are trivial as skew left braces. Then $A$ has a meta-trivial ideal $I$ such that $A/I$ is a trivial skew left brace.
\end{cor}

\begin{proof}
    By Lemma \ref{lem:hardworkfactoriz}, the ideal $A^{(2)}$ has a factorization through the strong left ideals $B*C$ and $C*B$,
    which are trivial skew left braces. Hence, by Theorem~\ref{thm:Ito2}, $A^{(2)}$ is meta-trivial. Thus the  claim follows.
\end{proof}

The assumptions of Theorem~\ref{thm:Ito2} cannot be relaxed. For the examples stated below we use~\GAP~and 
the database of small skew left braces of~\cite{MR3647970}, see~\cite[\S2.1]{KSV} for notation.

\begin{exa}
The skew left brace $S(36,271)$ has additive group isomorphic to $C_3 \times A_4$ and multiplicative group isomorphic to $C_6 \times S_3$ and can be 
factorized in left ideals $X$ and $Y$, which are trivial as skew left braces, where $|X|=4$ and $|Y|=9$. In this example, $X$ is a strong left ideal, 
but $Y$ is not. This skew left brace is not meta-trivial or left nilpotent. However, it is right nilpotent of class four.
\end{exa}

The condition in the statement of Theorem \ref{thm:Ito2} that both left ideals have to be normal in $(A,+)$ cannot be replaced
by normality in $(A,\circ)$:

\begin{exa}
The group $(G,+)=S_4$ can be exactly factorized by the subgroups $X=\langle (2 4 3)\rangle$ and
$Y=\langle(3 4), (1 3) (2 4), (1 4) (2 3)\rangle$. The operation 
\[
(x+y) \circ (x_1 +y_1) = (x +x_1)+(y_1 + y)
\]
turns $G$ into a skew left brace such that $G^{(2)}$ is not a trivial skew left brace. In this example,
$Y$ is a left ideal and $X$ and $Y$ are normal in $(G,\circ)$. 
\end{exa}

\begin{thm}
\label{thm:Itocor}
Let $A$ be a non-zero skew left brace that has a factorization $A=B+C$ through left ideals $B$ and $C$, where both are trivial as skew left braces.
If $B$ is a strong left ideal of $A$, then $B$ or $C$ contains a non-zero ideal $I$ of $A$ that acts trivially on $A$.
\end{thm}

\begin{proof}
By Lemma \ref{lem:hardworkfactoriz}, it holds that $B*C$ and $C*B$ are trivial skew left braces, strong left ideals of $A$ and $A^{(2)}=B*C+C*B$. 
Using the same calculation as in Theorem \ref{thm:Ito2}, it follows that $B*C$ acts trivially on $A$. We show that $(B*C,\circ)$ is a normal
subgroup of $(A,\circ)$. Let $b,\beta \in B$ and $c,\gamma \in C$. By Lemma \ref{lem:calculations}, it holds that 
\begin{align*} 
    c \circ b \circ (\beta * \gamma) \circ b' \circ c' & =c \circ ( b + \lambda_b ( \beta*\gamma + \beta * b') -b)\circ c'\\
                                                       & =c \circ (b + \lambda_b(\beta*\gamma) -b) \circ c'.
\end{align*}
As $B*C$ is a strong left ideal, it is sufficient to show that $B*C$ is closed under conjugation of $(C,\circ)$.
Let $b \in B$ and $c,\gamma \in C$. Then, as $C$ is a trivial skew left brace, it follows that 
\[
\gamma \circ b*c \circ \gamma' = \gamma + b*c - \gamma.
\]
As $(B*C,+)$ is normal in $(A,+)$, this shows that $B*C$ is an ideal of $A$. If this is non-zero, then the theorem  follows. If $B*C=0$,
then $A^{(2)} = C*B$. This shows that in this case $C*B$ is an ideal. If this is also zero, then $A^{(2)} = 0$ and hence $A$ is a trivial
skew brace, $B$ and $C$ are ideals of $A$.
\end{proof}

The conditions in the statement of Theorem \ref{thm:Itocor} cannot be relaxed:

\begin{exa}
There exists a skew left brace $A$ with $(A,+) \cong (C_3 \times C_3)\rtimes C_2$ and $(A,\circ) \cong C_3 \times S_3$. The skew brace $A$ can be 
written as the sum of two left ideals $X$ and $Y$, which are trivial as skew left braces, where $|X|=6$ and $|Y|=3$. However, no ideal of $A$ is 
contained in $X$ or $Y$.
\end{exa}

\begin{cor}
\label{cor:class4}
Let $A=B+C$ be a skew left brace with a factorization through the left ideals $B$ and $C$, which are trivial as skew left braces.
If $B$ is a strong left ideal of $A$, then $A$ is right nilpotent of class at most four.
\end{cor}

\begin{proof}
Notice that, by the proof of Theorem \ref{thm:Itocor}, $B*C$ is an ideal of $A$. Assume first that $B*C=0$. Then $B*A=0$. Moreover, 
$A^{(2)}=C*B\subseteq B$ by Lemma~\ref{lem:hardworkfactoriz}(3). Thus $A^{(3)}\subseteq B*A=0$. Assume now that $B*C\ne 0$.
Then the skew left brace $A/B*C$ is right nilpotent of class at most three by the previous case. Hence $A^{(3)} \subseteq B*C$. As, by the proof
of Theorem~\ref{thm:Ito2}, $B*C$ acts trivially on $A$, it follows that $A^{(4)}=0$.
\end{proof}

Our results on factorizations of skew left braces have applications to involutive non-degenerate solutions of the Yang--Baxter equation.
Let $(X,r)$ be an involutive non-degenerate solution and write $r(x,y)=(\sigma_x(y),\tau_y(x))$. On $X$ one considers the equivalence relation
given by \[x\sim y\iff \sigma_x=\sigma_y.\] Then $\sim$ induces an involutive non-degenerate solution on the set of equivalence classes of
$X$ known as the \emph{retraction} $\Ret(X,r)$ of $(X,r)$, see \cite{MR1722951}. One defines inductively:
\begin{align*}
    \Ret^0(X,r) & =(X,r),\\
    \Ret^1(X,r) & =\Ret(X,r),\\
    \Ret^n(X,r) & =\Ret(\Ret^{n-1}(X,r)),\quad n\geq 2.
\end{align*}

A solution $(X,r)$ is said to be \emph{multipermutation} of level $m$ if $\Ret^m(X,r)$ has size one and $\Ret^{m-1}(X,r)$ has more
than one element. Multipermutation solutions can be characterized in terms of left orderability of the structure group $G(X,r)$ of
$(X,r)$, see~\cite{MR3815290,MR3572046,MR2189580}. This group is defined as the group generated by the elements of $X$ with relations
\[
x\circ y=u\circ v\quad\text{whenever $r(x,y)=(u,v)$}.
\]
The group $G(X,r)$ is a left brace with $\circ$ as the multiplicative operation and the quotient left brace 
\[
\mathcal{G}(X,r)=G(X,r)/\Soc(G(X,r))
\]
has multiplicative group isomorphic to the subgroup of $\Sym(X)$ generated by all $\lambda_x$ for $x\in X$, for details see~\cite{MR3177933,CAGTA}.

It is known \cite[Proposition 6]{MR3574204} that a non-zero left brace $A$ is right nilpotent of class $m$ if and only if its
associated solution of the Yang--Baxter equation $(A,r_A)$ is a multipermutation solution of level $m-1$. It is also known
\cite[Theorem 5.15]{MR3861714} that an involutive non-degenerate (not necessarily finite) solution of the Yang--Baxter equation
$(X,r)$ with $|X|\geq 2$ is a multipermutation solution of level not exceeding $m$ provided the solution $(G,r_G)$, where $G=G(X,r)$,
is a multipermutation solution of level $m$. Both these results together with our Theorem~\ref{thm:Ito2} lead to an interesting
information on involutive non-degenerate solutions of the Yang--Baxter equation.

\begin{cor} 
\label{cor:multiperm}
Let $(X,r)$ be an involutive non-degenerate (not necessarily finite) solution of the Yang--Baxter equation with $|X|\geq 2$.
If the left brace $\mathcal{G}(X,r)$ admits a factorization through left ideals, which are trivial as left braces, then $(X,r)$ is a
multipermutation solution of level at most three.
\end{cor}

\begin{proof}
Let $A=\mathcal{G}(X,r)$ and $G=G(X,r)$. Then Theorem~\ref{thm:Ito2} yields $A^{(m)}=0$ for some $m\leq 3$. Because $G/\Soc(G)\cong A$
as left braces, we get $G^{(m)}\subseteq\Soc(G)$, and thus $G^{(m+1)}=0$. Hence $G$ is a right nilpotent left brace of class at most four and,
by \cite[Proposition 6]{MR3574204}, $(G,r_G)$ is a multipermutation solution of  level at most three. Therefore, by~\cite[Theorem 5.15]{MR3861714}, 
$(X,r)$ is a multipermutation solution of level at most three.
\end{proof}

This shows that properties of the involutive non-degenerate set-theoretic solution $(X,r)$ are not completely determined by the
group theory of the additive and multiplicative groups of the left brace $\mathcal{G}(X,r)$. The following examples clarify this.

\begin{exa}
\label{exa:B(8,27)}
Let $X=\{1,2,3,4\}$ and $r(x,y)=(\sigma_x(y),\tau_y(x))$ be the irretractable involutive non-degenerate solution given by 
\begin{alignat*}{4}
    \sigma_1 & =(34),\qquad & \sigma_2 & =(1324),\qquad & \sigma_3 & =(1423),\qquad & \sigma_4 & =(12),\\
    \tau_1 & =(24),\qquad & \tau_2 & =(1432),\qquad & \tau_3 & =(1234),\qquad & \tau_4 & =(13).
\end{alignat*}
The associated left brace $\mathcal{G}(X,r)$ is $B(8,27)$ and has additive group $C_2^3$ and multiplicative group $D_8$. Furthermore,
$\mathcal{G}(X,r)$ is not right nilpotent. Hence it is impossible to decompose the left brace $\mathcal{G}(X,r)$ as in Corollary~\ref{cor:multiperm}. 
\end{exa}

\begin{exa}
The left brace $B(8,26)$ has the same additive and multiplicative groups as the brace $\mathcal{G}(X,r)$ 
of Example~\ref{exa:B(8,27)} but it has a factorization as in Corollary~\ref{cor:multiperm}. 
This shows that $B(8,26)$ is right nilpotent. 
\end{exa}

As a counterpart to Theorem~\ref{thm:Ito2}, an analog of It\^{o}'s theorem on metabelian groups \cite{MR0071426}, it is natural to ask if we can extend 
this analogy. A natural candidate is the celebrated Kegel--Wielandt theorem \cite{MR0133365,MR0224703,wiel1958}, which states that a finite group which 
is factorized by two nilpotent subgroups is solvable. In the following example we show that we cannot expect a naive analog of this theorem for skew 
left braces. Recall that a skew left brace is said to be \emph{simple} if is contains no non-trivial proper ideals. 

\begin{exa}\label{ex:Burn}
Consider the left brace $B(72,475)$, a simple left brace of size $72$. This left brace 
admits a factorization through non-trivial left braces $B$ and $C$ 
with $|B|=9$ and $|C|=8$ that are both left and right nilpotent (strong) left ideals. 
\end{exa}

Moreover, as $72=2^3 3^2$, Example~\ref{ex:Burn} shows also that the $p^aq^b$-theorem of Burnside does not have an analog for skew left braces. 
This also follows directly from work of Ced\'{o}, Jespers, and Okni\'nski \cite{CJO18}. As the construction of simple left braces has received a lot of 
attention \cite{CJO18, MR3763276, MR3812099} and because every left brace is a matched product of left ideals \cite{MR3812099,MR3763276} (if it is not 
of prime power order) we mention the following:

\begin{cor}
Let $A$ be a skew left brace. If $A$ is the matched product of trivial skew left braces $B$ and $C$, then $A$ cannot be a simple skew left brace. 
\end{cor}

\begin{proof}
By definition of the matched product, it is clear that both $B$ and $C$ are strong left ideals of $A$ and that $A=B+C$.
Hence Theorem \ref{thm:Itocor} shows that $A$ will contain a non-trivial ideal.
\end{proof}

\begin{defn}
Let $A=B+C$ be a factorization of a skew left brace $A$ through left ideals $B$ and $C$.
A left ideal $I$ of $A$ is said to be \emph{factorized} if $I=(I\cap A)+(I\cap B)$.
\end{defn}

If $A$ is a skew left brace, then \[\Fix(A)=\{a\in A:\lambda_b(a)=a\text{ for all $b\in A$}\}\] is a left ideal of $A$. 

\begin{cor}
Let  $A=B+C$ be factorization of a skew left brace $A$ through left ideals $B$ and $C$.  If $B$ and $C$ are both trivial as skew left braces,
then $\Fix(A)$ is a factorized left ideal of $A$. 
\end{cor}

\begin{proof}
Suppose $b+c \in\Fix(A)$ for some $b\in B$ and $c\in C$. Then, for any $\beta \in B$, we have $b+c=\lambda_\beta(b+c)=b+\lambda_\beta(c)$ because 
$\lambda_{\beta}$ is a homomorphism and $B$ is a trivial skew left brace. Hence $\lambda_{\beta}(c) = c$. Because $A=C\circ B$ and $C$ is a trivial
skew left brace, it follows that $c \in \Fix(A) \cap C$. By symmetry, we have that $b \in \Fix(A) \cap B$.
\end{proof}

Recall that a \emph{left brace} is a skew left brace with
abelian additive group. For left braces, $\Soc(A)=\Ker\lambda$. 

\begin{cor}
\label{cor:Soc_factorized}
Let $A$ be a left brace. If $A=B+C$ is a factorization through the left ideals $B$ and $C$,
where both are trivial as skew left braces, then $\Soc(A)$ is a factorized ideal of $A$.
\end{cor}

\begin{proof}
Let $s \in \Soc(A)$. Then $s=b+c$, where $b \in B$ and $c\in C$. We show that $b\in\Ker\lambda$. For that purpose,
let $\gamma \in C$. As $C$ is a left ideal and trivial as skew left braces, 
\[
\lambda_{b}(\gamma) = \lambda_{b\circ \lambda_{b'}(c)}(\gamma) = \lambda_{b+c}(\gamma) = \lambda_s(\gamma)=\gamma. 
\]
Hence $b$ acts trivially on both $B$ and $C$ and therefore $b \in \Ker \lambda=\Soc(A)$. This implies that $c=s-b\in\Soc(A)$. 
\end{proof}

In~\cite{MR1178147}, it is shown that exact factorizations of groups produce non-trivial solutions of the Yang--Baxter equation.
Such factorizations produce skew left braces: let $G$ be a finite additive group with an exact factorization through subgroups $B$ and $C$,
i.e., $G=B+C$. The operation 
\[
x\circ y=b+y+c,
\]
where $x=b+c$, turns the group $G$ into a skew left brace 
with multiplicative group isomorphic to $B\times C$ and additive group isomorphic to $G$, see~\cite[Theorem 2.3]{MR3763907}.
We say that $(G,+,\circ)$ is the skew left brace obtained from this exact factorization.

\begin{pro}
\label{propexactfact}
    Let $A$ be a skew left brace obtained from an exactly factorizable group $A=B+C$. 
    The following statements hold:
    \begin{enumerate}
        \item $C$ is a left ideal of $A$.
        \item If $(C,+)$ is normal in $(A,+)$, then $C$ is an ideal of $A$.
        \item If $(B,+)$ is normal in $(A,+)$, then $B$ is an ideal of $A$. 
    \end{enumerate}
\end{pro}

\begin{proof}
    Let $a=b+c$ with $b\in B$ and $c\in C$. To prove that $C$ is a left ideal
    let $\gamma\in C$. Then
    \begin{align*}
         \lambda_a(\gamma) & =\lambda_{b+c}(\gamma)\\
                           & =-(b+c)+(b+c)\circ\gamma\\
                           & =-c-b+b+\gamma+c\\
                           & =-c+\gamma+c\in C.
    \end{align*}
    Now let $\beta\in B$. Since
    \begin{align*}
    \lambda_a(\beta) & =\lambda_{b+c}(\beta)\\
                     & =-(b+c)+(b+c)\circ\beta\\
                     & =-c-b+b+\beta+c\\
                     & =-c+\beta+c,
    \end{align*}
    the second claim follows.
\end{proof}

\begin{cor}
Let $A$ be a skew left brace obtained from an exactly factorizable additive group, say $A=B+C$. Assume that both
$(B,+)$ and $(C,+)$ are normal in $(A,+)$. If $(C,+)$ is abelian, then $A$ is right nilpotent of class at most three.
In particular, its associated solution is a multipermutation solution of level at most three.
\end{cor}

\begin{proof}
    By Proposition \ref{propexactfact}, $B$ and $C$ are in particular strong left ideals of $A$.
    The definition of the circle operation implies that $B$ is a trivial skew left brace. Since $(C,+)$ is abelian, 
    $c_1\circ c_2=0+c_2+c_1=c_1+c_2$ for all $c_1,c_2\in C$ and hence $C$ is a trivial skew left brace.
    By Theorem~\ref{thm:Ito2}, $A$ is right nilpotent of class at most three. 
\end{proof}


\section{Characteristic ideals}
\label{characteristic_simple}

We denote by $\BrAut(A)$ the automorphism group of a skew left brace $A$.

\begin{defn}
Let $A$ be a skew left brace. An ideal $I$ of $A$ is called \emph{characteristic} if $\sigma(I) = I$ for all $\sigma\in\BrAut(A)$. 
A skew left brace $A$ is called \emph{characteristically simple} if $0$ and $A$ are its only characteristic ideals.
\end{defn}
As for finite groups, characteristically simple finite skew left braces can be described as a factorization
into copies of a simple skew left subbrace.

\begin{thm} 
\label{thm:charsimple}
Let $A$ be a finite skew left brace. Then $A$ is characteristically simple if and only if there exists a simple skew left brace $S$
and a positive integer $n$ such that $A \cong S^n$. Moreover, $A$ is solvable if and only if $S$ is trivial, and in this case $A$ is trivial.
\end{thm}

\begin{proof}
Let $I$ be a minimal ideal of $A$. Let $I_1 = I$. If $I_1,\dotsc,I_j$ are defined, then define $I_{j+1}= \sigma(I)$ for some $\sigma\in\BrAut(A)$
such that $\sigma(I) \nsubseteq I_1 +\dotsb+ I_j$. As $A$ is finite and characteristically simple, it follows that this procedure stops and there
exists a positive integer $n$ such that $A= I_1+\dotsb+I_n$. As $I_j \cap I_k = 0$ for $j \neq k$, it follows that $A\cong I_1\times\dotsb\times I_n$. 
Moreover, for any $1\leq k \leq n$, it clearly holds that $I_k \cong I$. Thus $A \cong I^n$. Let $J$ be an ideal of $I$. As $I$ is a direct factor of 
$A$, it follows that $J$ is an ideal of $A$. Hence either $J =I$ or $J = 0$. Thus $I$ is a simple skew left brace.

Let $S$ be a simple skew left brace and $n$ be a positive integer. Let $A=S^n$ and $I$ be a non-zero characteristic ideal of $A$. Then the projection
on the $k$-th component $\pi_k\colon A \to S$ is a surjective skew left brace homomorphism. Hence $\pi_k(I)$ is an ideal in $S$. Thus either
$\pi_k(I)=S$ or $\pi_k(I) = 0$. Suppose that $\pi_k(I) = S$. Let $a=(a_1,\dotsc,a_n) \in I$ with $a_k \neq 0$. Then there exists an $s \in S$ such
that $\lambda_s(a_k) \neq 0$ or $s+a_k-s \neq 0$ or $s \circ a_k \circ s' \neq 0$. Denote by $c$ the element of $A$ where every entry is $0$ except
the $k$-th which is $s$. Then $\lambda_c(a) \neq a$ or $c+a-c \neq a$ or $c \circ a \circ c' \neq a$. Denote one of the elements different from $a$
by $d$. Then $d-a \neq 0$. Since $d-a$ only has a non-trivial element in the $k$-th component, it follows that the ideal generated by this element
is the full $k$-th component. This shows that $I=C_{j_1}+\dotsb+C_{j_t}$, where $C_l$ denotes the full $l$-th component. As the symmetric group
$S_n$ acts on $A$ by permuting the indices and $I$ is a characteristic ideal, it follows that $I = A$.

Let $A$ be a solvable and characteristically simple skew left brace. As $S$ is a skew left subbrace of $A$, it follows that $S$ is a solvable
simple skew left brace. Hence $S$ is a trivial skew left brace. On the other hand, the direct product of trivial skew left braces is trivial,
so $A$ is trivial (and thus certainly solvable).
\end{proof}

\subsection*{Acknowledgments}
The first author is supported in part by Onderzoeksraad of Vrije Universiteit Brussel and Fonds voor Wetenschappelijk
Onderzoek (Flanders), grant G016117. The second author is supported by Fonds voor Wetenschappelijk Onderzoek (Flanders),
grant G016117. The third author is supported by Fonds voor Wetenschappelijk Onderzoek (Flanders), via an FWO Aspirant-mandate.
The fourth author is supported in part by PICT 2016-2481 and UBACyT 20020170100256BA.
Vendramin acknowledges the support of NYU-ECNU Institute of Mathematical Sciences at NYU Shanghai.

\bibliographystyle{abbrv}
\bibliography{refs}

\end{document}